\documentclass[]{siamart220329}
\usepackage{amsmath}
\usepackage{amssymb}
\usepackage{graphicx}

\usepackage{tikz}
\usepackage{pgfplots}
\usetikzlibrary{external}

\usepackage[Publish]{externaltikz}
\usepackage{calc}

\newcommand{\R}{\mathbb{R}} 

\newtheorem{remark}{Remark}
\usepackage{cite}

\begin{document}

\title{Interface  layers and coupling conditions for discrete kinetic   models  on networks: a spectral approach}

\author{R. Borsche\footnotemark[1] 
	\and T. Damm\footnotemark[1] 
     \and  A. Klar\footnotemark[1] 
 \and Y. Zhou\footnotemark[2] }
\footnotetext[1]{RPTU  Kaiserslautern, Department of Mathematics, 67663 Kaiserslautern, Germany 
  (\{borsche, damm, klar\}@rptu.de)}
\footnotetext[2]{RWTH Aachen, Department of Mathematics, Aachen, Germany 	(zhou@igpm.rwth-aachen.de)}
 
\date{}


\maketitle

\begin{abstract}
	We consider kinetic and related macroscopic equations on networks. 
	A class  of linear kinetic BGK models is considered, where the    limit equation for small Knudsen numbers is given by   the wave equation.
	Coupling conditions for the macroscopic equations are obtained  from the kinetic coupling conditions via an asymptotic analysis near the nodes of the network and the  consideration of coupled  solutions of kinetic half-space problems. 
	Analytical results are obtained for a discrete velocity version of the coupled half-space problems. Moreover,
	an efficient spectral  method is  developed to solve the coupled  discrete velocity  half-space problems. 
	In particular, this allows to determine the relevant coefficients in the coupling conditions  for the macroscopic equations
	 from the underlying kinetic network problem. These coefficients correspond to  the so-called extrapolation length for kinetic boundary value problems.
	Numerical results show the accuracy and fast convergence of the approach. Moreover, a comparison of the kinetic solution on the network with the macroscopic solution is presented.
\end{abstract}




\section{Introduction}
\label{introduction}
Coupling conditions for macroscopic partial differential equations on networks have been defined in many works including, for example, conditions for  drift-diffusion equations, scalar  hyperbolic equations, and  hyperbolic systems like  the wave equation or Euler type models, see for example \cite{BGKS14,CC17, BNR14,BHK06a,BHK06b,CHS08,EK16,VZ09,BCG10,HKP07,CGP05}. 
In particular, in \cite{CGP05,HKP07} coupling conditions for scalar hyperbolic equations on networks are discussed and investigated.
The wave equation is treated in \cite{EK16,VZ09}, and general non-linear hyperbolic systems are considered, for example,  in \cite{BNR14,BHK06a,BHK06b,CHS08,BCG10,CHS08,G10}. 
On the other hand, coupling conditions for kinetic equations on networks have been discussed, for example, 
in \cite{FT15,HM09,BKKP16,BK18,ABEK22,holle}. 
In \cite{BKKP16} a first attempt to derive a coupling condition for a
macroscopic equation from the underlying kinetic model has been presented for the case of a kinetic equation for chemotaxis.
In \cite{BK18,ABEK22} more  general and more accurate approximate procedures have been presented and discussed for linear kinetic equations.
They are motivated by the classical procedure to find kinetic slip boundary conditions for macroscopic equations via the analysis of the 
kinetic layer \cite{BSS84,BLP79,G08,UTY03,Bab,Ber}
and 
based on an asymptotic analysis of the situation near the nodes. 

In the present paper we consider the same situation as in  \cite{BK18}. However, in contrast to \cite{BK18}, where an approximation procedure for the  coupling conditions based on a low order half-moment approach is obtained, we consider  here 
the full kinetic layer problem via a hierarchy of discrete velocity models.
To investigate the coupled layer problems analytically we employ results from
\cite{MR1722195} and  \cite{BenzoniSerre} for hyperbolic relaxation problems.
The numerical solution of the problem  is obtained by adapting a spectral approach from \cite{coron} to the network problem.

The paper is organized in the following way. 
In Section \ref{sec:equations} we discuss  the kinetic and macroscopic equations   and  classes of coupling conditions for these equations.
In Section \ref{kinlayer}     an 
asymptotic analysis of the kinetic equations near the nodes and resulting  kinetic layers at the nodes are discussed.
This leads to an abstract formulation of the coupling conditions for the macroscopic equations at the nodes  involving coupled kinetic half-space problems. 
In the following Section \ref{dvm} a  velocity discretization of the kinetic equation via kinetic discrete velocity models is considered and the associated  kinetic moment problem is given.
In Section \ref{disclayer} the discrete layer problem on an edge in moment coordinates is investigated and solved  up to the 
determination of  the eigenvalues of an associated symmetric positive definite matrix. 
Finally, in Sections \ref{coup}   the solution of the  kinetic equations  at the node are discussed analytically and numerically and  the macroscopic coupling conditions are obtained.
In particular, in Subsection \ref{wellposed} the solvability of the coupled half-space problem is investigated analytically. In Subsections  \ref{numcoupling} and \ref{nearnode} the numerical strategy to obtain the coefficients for the macroscopic coupling conditions and the limiting kinetic solution at the node is described. Subsection \ref{halfspacemarshak} gives a short review  of simple approximate methods to determine the coupling conditions and Subsection 
\ref{numerical1} discusses issues concerning the numerical implementation and gives numerical results for the coupling coefficients. 
Section \ref{unbound} contains the same steps for the case of an unbounded velocity domain.
Finally, Section \ref{network} presents a numerical comparison of  kinetic and macroscopic 
network solution.

\section{A kinetic model equation and coupling conditions}
\label{sec:equations}

In this section we consider a kinetic equation with bounded velocity space. In Section \ref{unbound} the case of an unbounded velocity space  will be considered.
As a prototypical example, we  consider  a  one-dimensional  linear kinetic  BGK model \cite{BGK54}  for the distribution function $f= f(x,v,t)$ with 
$x \in \mathbb{R}$ and $ v  \in [-1,1]$, i.e.
\begin{align}\label{bv}
\begin{aligned}
\partial_t f + v \partial_x f = \frac{1}{\epsilon} Q(f) = -\frac{1}{\epsilon} \left(  f-   M_f\right) =-\frac{1}{\epsilon} \left(  f-    \frac{1}{2}(\rho + \frac{v}{a^2}  q )\right)
\end{aligned}
\end{align}
with $ \epsilon >0$, $a^2 = \frac{1}{2} \int^1_{-1} v^2 d v  = \frac{1}{3}$ and
$$
\rho = \frac{1}{2} \int_{-1}^1 f(v) d v ,\quad  q = \frac{1}{2}  \int_{-1}^1 v f(v) d v \ .
$$
Integrating the equation with respect to $dv $ and $v dv$ and taking into account that $f $ converges towards $M_f$ as $\epsilon \rightarrow 0$, the associated macroscopic equation  for $\epsilon \rightarrow 0$ is
the wave equation 
\begin{align}\label{waveeq}
\begin{aligned}
\partial_t \rho_0 + \partial_x q_0 &=0\\
\partial_t q_0 + a^2 \partial_x  \rho_0  &=0\ .
\end{aligned}
\end{align}
Here, we have denoted the limiting macroscopic quantities for $\epsilon \rightarrow 0$, i.e. the solution of the macroscopic limit equations,
by $\rho_0,q_0$. Quantities $\rho,q$ without a subscript denote the kinetic density and mean flux.
The eigenvalues of system (\ref{waveeq}) are
$
\lambda_{\mp} =   \mp a.
$
The corresponding eigenvectors  are
$
\begin{pmatrix}
	1 ,  \mp a 
	\end{pmatrix}^T.
$

\begin{figure}[h!]
	\begin{center}
		\begin{tikzpicture}
		\draw[->] (0,0)--(-1.4,0) node[left]{$1$};
		\draw[->] (0,0)--(1,-1) node[right]{$2$};
		\draw[->] (0,0)--(1,1) node[right]{$3$};
		\node[fill=black,circle] at (0,0){};
		\end{tikzpicture} 
	\end{center}
	\caption{Node connecting three edges and orientation of the edges.}
	\label{fig:tripod}
\end{figure}
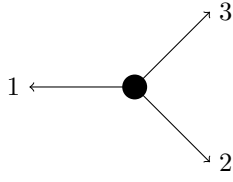
If these equations are considered on a network, it is sufficient to study a single node, see Figure \ref{fig:tripod}.
At each node so called coupling conditions are required. 
In the following we consider a node connecting $n$ edges, which are oriented away from the node, as in Figure \ref{fig:tripod}.
Each edge $i$ is parametrized by the interval $[0,b_i]$ and the kinetic and macroscopic quantities are denoted by
$f^i$ and $\rho_0^i, q_0^i$ respectively.
On the kinetic level for each edge a condition on  the ingoing characteristics $f^i(0,v),v>0$ is required at the node, i.e.
at   $x=0$.
For the network problem  a possible choice of such a coupling condition for the kinetic problem is
\begin{align}
\label{kincoup}
f^i(0,v) = \sum_{j=1}^n \beta_{ij}f^j (0,-v), v >0\ , i= 1, \ldots ,n, 
\end{align}
compare \cite{BKKP16}.
Then, the total mass in the system is conserved, if 
\begin{align}\label{eq:conservative_coupling_f}
\sum_{i=1}^n \beta_{ij} = 1,
\end{align}
since in this case the balance of fluxes,  i.e. $\sum_{j=1}^n \int_{-1}^1 v f^j(0,v) dv =0$, holds.
Note for later use, that we  have for odd moments in general $$\sum_{j=1}^n \int_{-1}^1 v^{2k-1} f^j(0,v) dv =0, k\ge 1. $$

In particular, we will consider the case of a node with symmetric coupling conditions, 
that means  $\beta_{ij} = \frac{1}{n-1}, i \neq j$  and $\beta_{ii} =0$.

In the macroscopic case, the  coupling conditions for the  system of linear hyperbolic equations are  conditions for the $2n$ macroscopic quantities $(\rho_0^i,q_0^i)(x=0)$  at the nodes. They   are given by $n$ 
coupling conditions
to find  the ingoing (into the adjacent edges) characteristic variables at the nodes
\begin{align*}
r_+^i (0) = q_0^i(0) + a \rho_0^i(0).
\end{align*}
If these  $n$ coupling conditions are combined 
with the  $n$  conditions given by the actual states of the outgoing characteristics $r_1(0)$ at the nodes, i.e. 
\begin{align*}
	q_0^i(0) - a \rho_0^i(0) = r^i_-(0) , i=1, \ldots, n,
\end{align*}
we obtain the required number of $2n$ conditions.
One of the coupling  conditions is usually 
 given by  the balance  of fluxes
$$\sum_{i=1}^n q_0^i (0) =0.$$
Note that this  condition corresponds to condition (\ref{eq:conservative_coupling_f}) on the kinetic level.
For symmetric nodes, 
further conditions are classically given by invariants at the nodes. For the present system of two equations,
we need one more invariant at the node leading to $n-1$ conditions for the macroscopic quantities.
 This invariant is usually given by a linear combination  $q_0^i(0) + \delta \rho_0^i (0)$.
In  other words, the missing $n-1$ equations are given by the conditions
\begin{align*}
 \rho_0^i(0) + \delta q_0^i(0) = \rho_0^j(0) + \delta q_0^j(0)   
\end{align*}
for $ i,j = 1, \ldots, n$. 
Together, these  macroscopic coupling conditions yield  a  linear system 
\begin{align} 
\label{macrocond}
\mathcal{B} U =b
\end{align} 
for $$U = (\rho_0^1(0), \ldots, \rho_0^n(0), q_0^1(0), \ldots
,q_0^n(0))^T $$ 

and 
\begin{align*} 
\mathcal{B} = \begin{pmatrix}
B_{11}&B_{12}\\
-aI &I
\end{pmatrix} \in \R^{2n\times 2n}
\end{align*}
with
\begin{align*} 
B_{11}= \begin{pmatrix}
0&0&0& \cdots &0\\
1&-1&0&\cdots &0\\
0&1&-1 & &0\\
\vdots& \ddots& \ddots&\ddots&\\0&\cdots &0&1&-1
\end{pmatrix},\;
B_{12}= \begin{pmatrix}
1&1&  1&\cdots &1\\
\delta&-\delta&0& \cdots &0\\
0&\delta&-\delta & &0\\\vdots & \ddots & \ddots & \ddots&\\0&\cdots &0&\delta&-\delta
\end{pmatrix} \in \R^{n \times n}\;.
\end{align*}
and
\begin{align*} 
b= \begin{pmatrix}
0& \cdots &0&r_1^1 (0)& \cdots r_1^n (0)
\end{pmatrix} \in \R^{2n}.
\end{align*}
We have a uniquely solvable system, if 
$$
0\neq \det(\mathcal{B}) 
= 
 \det(B_{11} +a  B_{12} )= (-1)^{n-1}na(1+a \delta)^{n-1}\;,   
$$
i.e., if $\delta \neq - \frac{1}{a}$.


The question naturally arises, how kinetic and macroscopic coupling conditions are connected and, in particular, if a  value  for  $\delta$ can be identified  associated to the kinetic coupling conditions 
 (\ref{kincoup}) in the asymptotic limit $\epsilon \rightarrow 0$, when the kinetic problem converges towards the macroscopic one.
 
 \begin{remark}
The   number  $\delta$ in the coupling conditions plays a similar role as the so-called extrapolation length for kinetic boundary layers, see \cite{BSS84}.
 \end{remark}
 
In \cite{BK18,ABEK22}, see also Section \ref{halfspacemarshak}, several approximation  procedures to obtain explicit formulas for the   values of $\delta$  have been proposed. 
In the present investigation   we aim at determining the value of $\delta$ for  the  full kinetic problem.  We investigate a numerical procedure for a hierarchy of   kinetic discrete velocity models to obtain a value for  $\delta$.  In this way   we obtain a very accurate approximation of the value corresponding to the continuous kinetic problem.


\section{Kinetic layers  at the nodes and coupling conditions for macroscopic equations}
\label{kinlayer}

The derivation of macroscopic coupling conditions from the kinetic conditions  is based on a kinetic layer analysis at the node, compare  \cite{BSS84,Bardos,C69,CGS88,G10,GMP,LLS16,coron,SO} for kinetic boundary value problems.
At the left boundary of  each edge $[0,b_i]$     a rescaling of the spatial variable in equation (\ref{bv}) with $\epsilon$ results  in  the scaled equation
$$
\partial_t f +  \frac{1}{\epsilon} v \partial_x f = \frac{1}{\epsilon} Q(f)\ 
$$
on $[0, \frac{b_i}{\epsilon}]$.
This yields to first order in $\epsilon$ the  following stationary
kinetic  half space  problem for 
the scaled spatial variable $x \in [0,\infty]$
\begin{align}\label{bgkhalfspace}
\begin{aligned}
v \partial_x \varphi =-\left(  \varphi-   \frac{1}{2}  (\rho + \frac{v}{a^2}  q )\right) \ ,
\end{aligned}
\end{align}
where $\rho$ and $q$ are  here the zeroth and first moments of  $\varphi$. 
At $x=0$ one has to prescribe for the  half space problem, as for the original kinetic problem,  the ingoing characteristics, i.e.
$$
\varphi(0,v), v > 0\ .
$$

For the coupling procedure we are only interested in bounded solutions of the half-space problem. Then, at  $x=\infty$,  a further condition is  needed for the half-space problem prescribing 
a linear combination of the invariants of the half-space problem
$\int_{-1}^1 v \varphi d v $
and $\int_{-1}^1 v^2\varphi dv $ .
The resulting solution of the half -space problem at infinity has the form 
$$
\varphi(\infty,v) =\frac{1}{2}  (\rho_\infty+ \frac{v}{a^2} q_\infty) ,
$$
where  $ \rho_\infty$ and $ q_\infty  $ are the corresponding density and mean flux of  the solution of the half-space problem solution at infinity.

The resulting outgoing  solution of the half space problem at $x=0$ is 
$$
\varphi(0,v),   v<0\ .
$$

In a classical matching procedure, the above solution at infinity of the half-space problem
is now  connected to the outer solution given by the macroscopic solution
at the left boundary of the edge 
 $(\rho_0(0) ,q_0(0))$.
 This means the missing condition for the half space problem  is given by the  1-Riemann invariant
 $$
 q_\infty -a \rho_\infty = q_0(0)-a \rho_0(0).
 $$
 In other words, we have the condition
 $$\frac{1}{2} \int_{-1}^1 \left(v - \frac{v^2}{ a } \right)\varphi d v= r_-(0)  
 $$
 at $x=\infty$ for the half-space problem. 
%
%
%
Solving  then the half-space problem gives 
 $\rho_\infty, q_\infty $ and thus
$$
q_0(0) + a \rho_0(0)= q_\infty + a \rho_\infty,
$$
which are   the required values for the ingoing characteristics of the  macroscopic equations at the nodes.

%
%
%
%

We combine now the layers on all edges adjacent to the node under consideration and use the kinetic coupling conditions
to obtain
\begin{align*}
\varphi^i(0,v) = \sum_{j=1}^n \beta_{ij}\varphi^j (0,-v), v >0.
\end{align*}
This gives the equations for the ingoing solutions of the half space problems on the different
arcs.
To conclude, finding the macroscopic coupling conditions associated to the underlying kinetic problem is equivalent to solving the above described coupled kinetic half-space problems on all edges of a node.
In the following sections we will consider a velocity discretized version of the kinetic problem and discuss the analytical and numerical solution of the coupled half-space problems and the resulting macroscopic coupling conditions in detail. 

%

\section{The discrete velocity model}
\label{dvm}
We discretize the BGK-equation (\ref{bv}) in velocity space and obtain a kinetic discrete velocity model  for the discrete distribution functions 
$f_i(x,t), i= 1, \ldots,  2N$ as 

\begin{align}\label{eq:kinetic}
\begin{aligned}
	\partial_t f_i + v_i\partial_xf_i &= -\frac{1}{\epsilon}\left(f_i-M_i\right)
\end{aligned}
\end{align}
with  the velocity discretization $$-1\le v_1< v_2 < \cdots < v_N < 0 < v_{N+1}< \cdots < v_{2N-1} < v_{2N} \le 1.$$
We  assume for  symmetry 
\begin{align*}
v_{2N} = -v_1, 
\ldots, 
v_{N+1}= -v_N.
\end{align*}
Let $w_i \ge 0, i=1, \ldots, 2N$ be symmetric weights, such that $\ \sum_{i=1}^{2N}  w_i =1$.
The  discrete  linearised Maxwellian $M_i$ is given by
\begin{align}
M_i = w_i (\rho + \frac{v_i}{a_N^2} q)
\end{align}
with 
\begin{align*}
\rho = \sum_{i=1}^{2N}f_i\ ,
q = \sum_{i=1}^{2N}v_if_i\ 
\end{align*}
and  $a^2_N= \sum_{i=1}^{2N} w_i v_i^{2} $.
This choice of the discrete Maxwellian yields 
\begin{align}
\label{mommax}
\sum_{i=1}^{2N}M_i=\rho\ ,\; 
\sum_{i=1}^{2N}v_iM_i = q\ , \;
\sum_{i=1}^{2N}v_i^2 M_i =a_N^2 \rho 
\end{align}
and we obtain for  equation (\ref{eq:kinetic}) in  the limit $\epsilon \rightarrow 0$   the wave equation
\begin{align}
\label{euler11}
\begin{aligned}
\partial_t \rho + \partial_xq &= 0\\
\partial_t q +  a_N^2   \partial_x \rho  &= 0.
\end{aligned}
\end{align}
Continuing, we define, additionally to $\rho $ and $q$, the  moments
\begin{align*}
	g_j = \sum_{i=1}^{2N}P_j(v_i) f_i\ , j	= 0, \ldots, 2N-1
\end{align*}
for some basis  $P_0, \ldots, P_{2N-1}$ of the space of polynomials up to degree $2N-1$, where $P_0$ is a multiple of $1$
and $P_1$ a multiple of $v$.
Let $g$ be given by  $g = (g_2, g_3,\ldots,g_{2N-1}).$ 

The transformation from original to moment variables is given by the  Vandermonde like  matrix
\begin{align*}
	S  
	=\begin{pmatrix}
		P_0(v_1) &\cdots&P_0(v_{2N})\\
		\vdots&&\vdots\\
		P_{2N-1}(v_1)&\cdots&P_{2N-1}(v_{2N})\\
	\end{pmatrix}	
\end{align*}
with $S \in \R^{2N \times 2N}$ 
transforming  the  variables $f=(f_1, \ldots, f_{2N})^T$ into the moments $ Sf = G=\begin{pmatrix}g_0,g_1  , \ldots, g_{2N-1} \end{pmatrix}^T.$

\begin{remark}
In  principle any choice of the discretization points $v_i$ and the polynomials $P_i$ could be used.
However, the situation simplifies considerably, if a  suitable orthonormal polynomial system and the associated discretization points are used. Moreover, from a numerical point of view, such a choice guarantees that the matrix $S$ is not ill conditioned.
An arbitrary choice, like, for example, equidistantly distributed points $v_i$ and a monomial basis or also 
 equidistantly distributed points combined with orthonormal polynomials will lead to strongly ill-conditioned matrices $S$
 for larger values of $N$.
	\end{remark}
%


For  the following we 	choose as in the works of F.\ Coron \cite{coron}  the $P_j $ as the normalized Legendre polynomials on $[-1,1]$.
The  discretization points  $v_i, i=1, \ldots, 2N$ are chosen as the associated  Gauß-Legendre points on $[-1,1]$ and $w_i$ the corresponding weights, such that 
$$
\sum_{i=1}^{2N} w_i P_j (v_i)P_k(v_i) = \delta_{jk}.
$$

The   orthonormal  Legendre polynomials $P_k = P_k(v), k=0, \ldots ,2N$ on $[-1,1]$ are defined via $P_0=\frac{1}{\sqrt{2}},P_1 =  \sqrt{\frac{3}{2}}v=\frac{1}{\sqrt{2} \alpha_1}v $ and the recursion formula
\begin{align*}
v P_k = \alpha_{k+1} P_{k+1} + \alpha_k P_{k-1}, k=1, \ldots, 2N-1
\end{align*}
with $\alpha_k = \frac{k}{\sqrt{(2k-1)(2k+1)}}$.
In particular, $P_2=\frac{1}{\alpha_1 \alpha_2 \sqrt{2}} (v^2-\alpha_1^2)=\sqrt{\frac{5}{8}}(3 v^2-1)$.


We have $g_0= \frac{\rho}{\sqrt{2}}$ and $g_1 = \frac{1}{\sqrt{2} \alpha_1 } q$. Moreover, for $k=2, \ldots, 2N-1$ the additional discrete moments of the Maxwellian, i.e.
$$\sum_{i=1}^{2N}P_{k}(v_i)  M_i $$ can, in general, be computed as functions of $\rho$ and $q$.
	Using Legendre polynomials and the associated Gauß-Legendre points all these higher order discrete moments  of the Maxwellian
	are equal to $0$ due to discrete orthogonality.
	Moreover, 
 note that the $2N$-th moment
\begin{align*}
	g_{2N}=\sum_{i=1}^{2N}P_{2N} (v_i)f_i\ 
\end{align*}
is also equal to zero, since the Gauß-Legendre points are   the zeros of the $2N$-th Legendre polynomial.
Finally, 
note that $a_N^2=a^2=\frac{1}{3}= \alpha_1^2 $ and 
 that $$g_2 =\frac{1}{\alpha_1 \alpha_2 \sqrt{2}}( \sum_{i=1}^{2N} v_i^2f_i -\alpha_1^2 \rho) =  \sqrt{\frac{5}{8}}(3 \sum_{i=1}^{2N} v_i^2f_i -\rho)$$ and therefore $$\sum_{i=1}^{2N} v_i^2f_i= 
 \frac{2}{3} \sqrt{\frac{2}{5}}g_2 + \frac{1}{3} \rho
 =\alpha_1 \alpha_2 \sqrt{2} g_2 + \alpha_1^2 \rho.$$


Using the recursion formula of the Legendre polynomials 
the discrete  kinetic equation is rewritten in moment variables $ G= (u,g) $ with $u=(g_0,g_1)$ and $g= (g_2,\ldots,g_{2N-1}) $.
%
In case  the points $v_i$ are chosen as the Gauß-Legendre points on $[-1,1]$ we obtain
\begin{align}\label{eq:macro_4eq_h}
\begin{aligned}
\partial_t g_0 + \alpha_1 \partial_x g_1 &= 0\\
\partial_t g_1 + \partial_x ( \alpha_2  g_2 + \alpha_1 g_0)&= 0\\
\partial_t g_{2} + \partial_x (\alpha_{3} g_{3} + \alpha_2 g_1) &= - \frac{1}{\epsilon} g_{2}\\
\partial_t g_{k} + \partial_x (\alpha_{k+1}g_{k+1} + \alpha_{k} g_{k-1})&= - \frac{1}{\epsilon}g_{k} , k=3, \ldots, 2N-2\\
\partial_t g_{2N-1}+ \partial_x (\alpha_{2N-1}g_{2N-2})&= - \frac{1}{\epsilon}g_{2N-1} \end{aligned}.
\end{align}

%
or for the first 3 equations
\begin{align}
	\begin{aligned}
		\partial_t \rho + \partial_xq &= 0\\
		\partial_t q + \partial_x (\alpha_1 \alpha_2 \sqrt{2} g_2 + \alpha_1^2 \rho)&= 0\\
		\partial_t g_{2} + \partial_x (\alpha_{3} g_{3} + \frac{\alpha_2}{\alpha_1 \sqrt{2}} q) &= - \frac{1}{\epsilon} g_{2}
	 \end{aligned}.
\end{align}
Note that  for this system we obtain in  the limit $\epsilon \rightarrow 0$ directly  the wave equation (\ref{waveeq}).

  \section{The discrete layer problem}
\label{disclayer}

The discrete kinetic half-space problem
\begin{align}
\begin{aligned}
 v_i\partial_xf_i &= -\left(f_i-M_i\right)
\end{aligned}
\end{align} is then transformed into 
the  moment layer equations
\begin{align}
\begin{aligned}
\alpha_1  \partial_x g_1&= 0\\
 \partial_x (\alpha_2 g_2 + \alpha_1 g_0 )&= 0\\
 \partial_x (\alpha_{3} g_{3} ) &= - g_{2}\\
 \partial_x (\alpha_{k+1}g_{k+1} + \alpha_{k} g_{k-1})&= - g_{k} , k=3, \ldots, 2N-1\\
 \partial_x (\alpha_{2N-1}g_{2N-2})&= - g_{2N-1}. \end{aligned}
\end{align}

This gives directly  $q=C $ and $\rho + \frac{\alpha_2 \sqrt{2}}{\alpha_1}g_2= D$ for constants $C \in \R$ and $D \in \R^+$.
For $g=(g_2, \ldots, g_{2N-1})$ we have 
\begin{align}
\begin{aligned}
\partial_x (\alpha_{3} g_{3} ) &= - g_{2}\\
\partial_x (\alpha_{k+1}g_{k+1} + \alpha_{k} g_{k-1})&= - g_{k} , k=3, \ldots, 2N-1\\
\partial_x (\alpha_{2N-1}g_{2N-2})&= - g_{2N-1}. \end{aligned}
\end{align}


In matrix form we have in case of Legendre polynomials with Gauss-Legendre points
\begin{align}
\label{dischalf}
 \partial_x g = - A_{22}^{-1} g 
\end{align}
with the symmetric tridiagonal matrix $A_{22} \in \R^{2(N-1) \times 2(N-1)}$ given by 
\begin{align}\label{eq:A_tridiag}
A_{22} = \begin{pmatrix}
0 & \alpha_3& 0&\cdots &\cdots &0\\ 
\alpha_3 &0&\alpha_4 & \ddots&&\vdots  \\
0&\ddots &\ddots&\ddots &\ddots &\vdots \\
\vdots&\ddots&\ddots &\ddots&\ddots &0\\
\vdots&  &\ddots&\ddots
&0 &\alpha_{2N-1}\\
0& \cdots &\cdots&0 &\alpha_{2N-1}&0
\end{pmatrix} .
\end{align}


The fixed point of the linear ODE  system  (\ref{dischalf}) is given by
$g=0$  and then $\rho =D$ and $q=C$.

\begin{lemma} 
	\label{lemma1}
	 $A_{22}$ is strictly hyperbolic, that means it is diagonalizable with real and distinct eigenvalues. Moreover,  $N-1$ eigenvalues of $A_{22}$ are strictly positive. The remaining $N-1$ eigenvalues have the corresponding negative values.
We denote the eigenvectors associated to positive eigenvalues by $r_i, i=1,\ldots, N-1$ 
and the matrix of those eigenvectors as  $$R_2^+= \begin{pmatrix} r_1,\ldots,r_{N-1}\end{pmatrix}.$$


	\end{lemma}

\begin{remark}
For a more general choice of discretization points we observe that the discrete layer problem is more complicated. In particular, the resulting matrix $A_{22}$  is not any more tridiagonal and the linear system is not homogeneous.
However,  a  Lemma  similar to Lemma \ref{lemma1} can still be proven for example in the case of equidistant points 
and monomials $P_j(v) = v^j, j=0, \ldots, 2N-1$.
\end{remark}

The full discrete boundary layer problem as  described in Section \ref{kinlayer} for the continuous case is then given 
in the variables  $G = (u,g)$ with $u=(g_0,g_1)$ and $g= (g_2, \ldots,g_{2N-1})$  as
\begin{equation}\label{intro}
A \partial_xG=QG
\end{equation}
with 
$$
A =\begin{pmatrix}
A_{11} & A_{12} \\
A_{21} & A_{22}
\end{pmatrix},\quad Q=\begin{pmatrix}
0 & 0 \\
0 & -I_{2N-2}
\end{pmatrix}.
$$
where
$$
A_{11}=\begin{pmatrix}
0 & \alpha_1\\
\alpha_1 & 0
\end{pmatrix},\quad
A_{12}=\begin{pmatrix}
0 & 0 & \cdots & 0\\
\alpha_2 & 0 & \cdots & 0
\end{pmatrix},\quad A_{21}=A_{12}^T.
$$

From the equation and the matching principle, we have
$$
u(x)+A_{11}^{-1}A_{12}g(x) \equiv u(\infty) := u_\infty.
$$
For the outer solution, we write in terms of the characteristic variables 
$$
u_\infty = \begin{pmatrix}
g_0\\
g_1
\end{pmatrix} (\infty)= R_1^+ \beta_+ + R_1^- \beta_-,
$$
where $R_1^+$ and $R_1^-$ are eigenvectors associated with positive and negative eigenvalues of $A_{11}$ and $\beta_{\mp} \in \R$. Specifically, we take $R_1^+=(1,1)^T$ and $R_1^-=(1,-1)^T$. It follows that $g_0(\infty)=\beta_++\beta_-$ and $g_1(\infty)=\beta_+-\beta_-$.
Note that for the original characteristic variables we have 
 $$\begin{pmatrix}
r+\\
r_-
\end{pmatrix} = \begin{pmatrix}
q_\infty+a \rho_\infty\\
q_\infty-a \rho_\infty
\end{pmatrix} = 2 \sqrt{2} \alpha_1 \begin{pmatrix}
\beta +\\
- \beta_-
\end{pmatrix} .$$


According to Lemma  \ref{lemma1} 
problem (\ref{dischalf})  is a linear dynamical system with fixed point $g=0$ and 
an associated stable manifold spanned by the eigenvectors  associated to the positive eigenvalues of $A$. 
To obtain a bounded solution of the discrete kinetic half space problem  the initial values at $x=0$, i.e. 
$g(0) = (g_2(0), g_3(0),\ldots , g_{2N-1}(0))^T$  have to be  located  in  this   manifold  spanned by the eigenvectors.  That means $g$ has  to fulfill
\begin{align*}
g(0) =\gamma_1 r_1 + \cdots + \gamma_{N-1} r_{N-1} = R_2^+ \gamma
\end{align*}
for $\gamma =(\gamma_1, \ldots,\gamma_{N-1})^T$ with some real values $\gamma_1, \ldots, \gamma_{N-1}$. 
 Using these considerations, we have 
\begin{align}
\label{U0}
G(0,t)=
R_\infty
\begin{pmatrix}
\beta_+\\[1mm]
\gamma 
\end{pmatrix}(0,t) + \beta_-  \begin{pmatrix}
R_1^-\\[1mm]
0
\end{pmatrix}\quad \text{with} \quad R_{\infty}=\begin{pmatrix}
R_1^+ & -A_{11}^{-1}A_{12}R_2^+\\[1mm]
0 & R_2^+ 
\end{pmatrix}.
\end{align}
For the boundary layer equation \eqref{intro} in moment variables with general  boundary condition
$$
B G(0,t)=b(t)
$$
  with $b(t)$ given and $B \in \R^{2N\times 2N}$,
solvability means that $\beta_+$ and $\gamma$ can be uniquely determined from the boundary condition for given $\beta_-$. Namely, the matrix $BR_{\infty}$ is invertible. Note for later use that 
the expression for $G(0,t)$ and $f(0,t)$  can be rewritten as
\begin{align}
\label{f0}
G (0,t) 
=	 T \begin{pmatrix}
D,C,\gamma
\end{pmatrix}^T,\;\; f(0,t)=S^{-1}  T \begin{pmatrix}
D,C,\gamma
\end{pmatrix}^T
\end{align}
with $T \in \R^{2N \times (N+1)}$ given by 
\begin{align*}
T=	\begin{pmatrix}
T_{11} & T_{12}\\[1mm]
0 & R_2^+ 
\end{pmatrix}
\end{align*}
with
\begin{align*}
T_{11}= \frac{1}{\sqrt{2}} \begin{pmatrix}
1&0  \\
0&\frac{1}{\alpha_1 }
\end{pmatrix}, \quad
T_{12}=	- \frac{\alpha_2}{\alpha_1}\begin{pmatrix}
e_1^T R_2^+ \\
0
\end{pmatrix}	,	
\end{align*}
where $e_1^T =(1,0,\ldots,0)$ is the unit vector in $\R^{2(N-1)}$.
$T$ is full rank, due to the linear independence of the eigenvectors. 

\section{The coupled half-space problems}
\label{coup}
The above discussion is now used  together with the discrete version of the kinetic coupling conditions (\ref{kincoup})
\begin{align}
	f^i_k(0) = \sum_{j=1}^n \beta_{ij}f^j_k (0),  i=1, \ldots, n, \; k=1, \ldots, N
\end{align}
to find the macroscopic coupling conditions at the  nodes.
In general, using the above expression (\ref{U0}) for $G^i(0,t)$ and $ S f^i(0,t)=  G^i(0,t)$ in the kinetic coupling conditions gives $n N$ equations for $nN$ unknowns $\beta_+^i,\gamma^i$ 
assuming $\beta_-^i$ is known.
Equivalently, using  (\ref{f0}) 
gives $n N$ equations for $n(N+1)$ unknowns $D^i,C^i,\gamma^i$.
%
The remaining $n$ equations are in this case  obtained from 
%
%
%
%
\begin{align}
\label{charac}
C^i - a D^i = q^i_\infty - a \rho^i_\infty = q^i_0(0) - a \rho^i_0(0) , i=1, \ldots ,n.
\end{align}

For further analytical and numerical results, we simplify the situation to the case of symmetric coupling conditions.
In case of fully symmetric coupling conditions with $\beta_{ij} = \frac{1}{n-1}, i \neq 0$  and $\beta_{ii} =0$ 
the complexity can be strongly reduced. 
Note first that the coupling conditions 
\begin{align*}
f^i(0,v) = \frac{1}{n-1}\sum_{l=1, l\ne i }^n f^l (0,-v), v >0\ , i=1, \ldots, n
\end{align*}
give for $v>0$ and $i\ne j$
\begin{align*}
(n-1) f^i(0,v) &= \sum_{l=1, l\ne i }^n f^l (0,-v)\\
&=\sum_{l=1, l\ne j }^n f^j (0,-v)+f^j(0,-v) - f^i(0,-v)\\
&= (n-1) f^j (0,v)+f^j(0,-v) - f^i(0,-v).
\end{align*}
Thus,
$$
(n-1) f^i (0,v)+f^i(0,-v) 
$$
is a kinetic  invariant at the nodes and we obtain  for the discretized equations $N$ invariants  at the nodes
\begin{align}
\label{invariant}
Z_1&=(n-1)f_{N+1}(0)+f_N(0)\nonumber\\
Z_k&=(n-1)f_{2N-k+1}(0)+f_i(0), k=2, \ldots, N-1\\
Z_N&=(n-1)f_{2N}(0)+f_1(0). \nonumber
\end{align}

%
%

Moreover, we have obviously 
\begin{align*}
\sum_{j=1}^n f^j(0,v) = \sum_{j=1}^n  f^j (0,-v), v >0
\end{align*}
and the corresponding discrete version
\begin{align*}
\sum_{j=1}^n f_{2N-k+1}^j(0) = \sum_{j=1}^n  f^j_k (0), \;  k=1, \ldots N.
\end{align*}
Alltogether we obtain  the kinetic coupling conditions in the following form
\begin{equation}\label{coupling-DV_old}
\begin{pmatrix}
 B_1& B_1  & \cdots & \cdots &    B_1\\[2mm]
 B_2 & - B_2& 0 & \cdots & 0\\[2mm]
\vdots & \vdots & \vdots & \vdots & \vdots \\[2mm]
 B_2  & 0 & \cdots & 0 & - B_2
\end{pmatrix}
\begin{pmatrix}
f^1(0,t)\\[2mm]
f^2(0,t)\\[2mm]
\vdots\\[2mm]
f^n(0,t)
\end{pmatrix}=0,
\end{equation}
where $ B_1=  (\hat I_N,-I_N)$ and $ B_2= (\hat I_N,(n-1)I_N) $. Here $I_N$  is the unit matrix and $$\hat I_N = 
\begin{pmatrix}
0 & \cdots & 0 &  0&1\\
0 & \cdots &0&1&0\\
&&\vdots &&\\
1 & 0&0&\cdots&0
\end{pmatrix}.$$
Using then $  f^i = S^{-1} U^i , i= 1, \ldots,n $ and expression (\ref{U0}) or (\ref{f0}) one obtains the coupled half-space problem 
as a linear system for $\beta_+^i, \gamma^i $ given $\beta_-^i$. Alternatively, this gives, with the additional equations $C^i-a D^i = r_-^i$, a  linear system for $C^i,D^i,\gamma^i $.

\subsection{Well-posedness  of the coupled half-space problem}
\label{wellposed}
We consider the coupled half-space problem described above and prove

\begin{theorem}
	\label{maintheorem}
	The coupled half-space problem is uniquely solvable for    given values of  characteristics  $r_-^i, i=1, \ldots, n$ on all edges,
	where $n \ge 3$.
\end{theorem}
\textbf{Proof of Theorem \ref{maintheorem}:}
Using an inverse  reordering  of the negative discrete velocities $v_i$ and the corresponding ordering of the $f_i, \; i=1, \ldots, 2N $, i.e. an ordering of the velocities  as 
$$ (v_N, v_{N-1},  \ldots , v_1,v_{N+1}, \ldots, v_{2N}),
$$
the above kinetic coupling conditions are written  as
\begin{equation}\label{coupling-DV}
\mathcal{B} \begin{pmatrix}
	f^1(0,t),
	f^2(0,t),
	\ldots,
	f^n(0,t)
\end{pmatrix}^T=0, \; \mathcal{B}= \begin{pmatrix}
B_1 & B_1 & \cdots & \cdots & B_1\\[2mm]
B_2 & -B_2 & 0 & \cdots & 0\\[2mm]
\vdots & \vdots & \vdots & \vdots & \vdots \\[2mm]
B_2 & 0 & \cdots & 0 & -B_2
\end{pmatrix},
\end{equation}
where in the reordered case $B_1=(I_N,-I_N)$, $B_2=(I_N,(n-1)I_N)$. 
Remark that we use in the proof  for  the reordered quantities the same notation as for the original ones.
Using then $G^j=S f^j$ with the reordered Vandermonde matrix 
	\begin{equation}\label{VS}
	S=\begin{pmatrix}
	P_0(v_N) & \cdots & P_0(v_{1}) & P_0(v_{N+1}) & \cdots & P_0(v_{2N}) \\[1mm]
	\vdots & \ddots & \vdots & \vdots & \ddots & \vdots \\[1mm]
	P_{2N-1}(v_N) & \cdots & P_{2N-1}(v_{1}) & P_{2N-1}(v_{N+1}) & \cdots & P_{2N-1}(v_{2N})
	\end{pmatrix}
	\end{equation}
the coupling condition \eqref{coupling-DV} is equivalent to 
	$$
	\mathcal{B} \begin{pmatrix}
	S^{-1} G^1(0,t),
	S^{-1} G^2(0,t),
	\ldots,
	S^{-1} G^n(0,t)
	\end{pmatrix}^T=0.
	$$

Using (\ref{U0}) and a direct computation one observes, that showing the solvability of the coupling problem
  is equivalent to checking the invertibility of $B_1S^{-1}R_\infty$ and $B_2S^{-1}R_\infty$. In other words, we need to check the solvability of the following two sub-problems:
$$
(\text{Problem}~1)~\left\{\begin{array}{l}
A\partial_xG = QG \\[3mm]
B_1S^{-1}G(0,t)=0
\end{array}\right.
\quad 
(\text{Problem}~2)~\left\{\begin{array}{l}
A\partial_xG = QG \\[3mm]
B_2S^{-1}G(0,t)=0.
\end{array}\right.
$$

\textbf{Problem 1:} 
It is not difficult to see that 
\begin{equation}
g_1(x)\equiv g_1(\infty),\qquad g_0(x)+\frac{\alpha_2}{\alpha_1}g_2(x)\equiv g_0(\infty).
\end{equation}
By introducing
$$
\bar{A}=
\begin{pmatrix}
\alpha_3 & 0 &  \cdots & 0\\
\alpha_4 & \alpha_5 & \ddots & \vdots\\
& \ddots & \ddots & 0\\
& & \alpha_{2N-2} & \alpha_{2N-1}
\end{pmatrix},\quad
g_{e} = \begin{pmatrix}
g_2\\
g_4\\
\vdots\\
g_{2N-2}
\end{pmatrix},\quad 
g_u = \begin{pmatrix}
g_3\\
g_5\\
\vdots\\
g_{2N-1}
\end{pmatrix},
$$
we rewrite the ODE for $g$ according to the even-odd partition
$$
\partial_x
\begin{pmatrix}
0 & \bar{A}\\[1mm]
\bar{A}^T & 0
\end{pmatrix}
\begin{pmatrix}
g_e\\[1mm]
g_u
\end{pmatrix}=
-\begin{pmatrix}
g_e\\[1mm]
g_u
\end{pmatrix}.
$$
It means that 
$$
\partial_x
\begin{pmatrix}
g_e\\[1mm]
g_u
\end{pmatrix}=
-\begin{pmatrix}
0 & \bar{A}^{-T}\\[1mm]
\bar{A}^{-1} & 0
\end{pmatrix}
\begin{pmatrix}
g_e\\[1mm]
g_u
\end{pmatrix}.
$$
For the coefficient matrix of this ODE system, we have, see, for example, \cite{LY23},
\begin{lemma}
	There exists an orthogonal matrix $\bar{R}$ such that $$
	\bar{R}^T\begin{pmatrix}
	0 & \bar{A}^{-T}\\[1mm]
	\bar{A}^{-1} & 0
	\end{pmatrix}\bar{R}=\begin{pmatrix}
	\Lambda_+ & 0\\[1mm]
	0 & -\Lambda_+
	\end{pmatrix},\quad 
	\bar{R}=\begin{pmatrix}
	\bar{R}_1 & \bar{R}_1\\[1mm]
	\bar{R}_2 & -\bar{R}_2
	\end{pmatrix}.
	$$
	Here $\Lambda_+$ is a diagonal matrix with positive entrances and $\bar{R}_1^T\bar{R}_1=\bar{R}_2^T\bar{R}_2=\frac{1}{2}I_{N-1}$.
\end{lemma}

\noindent Due to this, we write 
$$
\begin{pmatrix}
g_e\\[1mm]
g_u
\end{pmatrix}(0)=
\begin{pmatrix}
\bar{R}_1 \\[1mm]
\bar{R}_2 
\end{pmatrix}\gamma.
$$

\begin{lemma}\label{S}
For the reordered Vandermonde like matrix $S$ defined by \eqref{VS}, we have 
		$$
		S^{-1}=\begin{pmatrix}
		W & \\
		& W
		\end{pmatrix}S^T.
		$$
		Here $W$ is an $N\times N$ diagonal matrix with positive entrances..
\end{lemma}
\begin{proof}
	For the Gaussian–Legendre nodes  $v_1,\cdots,v_N,v_{N+1},\cdots,v_{2N}$, we take the symmetric Gaussian quadrature weights $w_1,w_2,...,w_N,w_{N+1}, \ldots, w_{2N}$  with $w_1=w_{2N},...,w_N=w_{N+1}$ and compute 
	\begin{align*}
	\sum_{k=1}^{N}w_{N+k}[P_i(v_{N-k+1})P_j(v_{N-k+1})+P_i(v_{N+k})P_j(v_{N+k})] 
	&=\delta_{ij}.
	\end{align*}
	Due to the above relation, we see that
	$$
	S\begin{pmatrix}
	W & \\
	& W
	\end{pmatrix}S^T = I_{2N}
	$$
	with $W=\text{diag}(w_{N+1},w_{N+2},...,w_{2N})$. This completes the proof of the lemma.
\end{proof}
Thanks to this lemma, we have with $B_1S^{-1} = W(I_N,-I_N)S^{T} $
$$
B_1S^{-1} =  -2W\begin{pmatrix}
0 & P_1(v_{N+1}) & 0 & P_3(v_{N+1}) & \cdots & 0 & P_{2N-1}(v_{N+1}) \\[1mm]
\vdots & \vdots & \vdots & \vdots & \vdots & \vdots &\vdots\\[1mm]
0 & P_1(v_{2N}) & 0 & P_3(v_{2N}) & \cdots & 0 & P_{2N-1}(v_{2N})
\end{pmatrix}.
$$
Note that we have used the relation $P_k(-v)=P_k(v)$ for even number $k$ and $P_k(-v)=-P_k(-v)$ for odd $k$. Then the boundary condition in Problem 1 becomes
\begin{equation}\label{BC1}
-2WS_u
\begin{pmatrix}
g_1\\
g_u
\end{pmatrix}(0)
=0
\end{equation}
with
$$
S_u=
\begin{pmatrix}
P_1(v_{N+1}) & P_3(v_{N+1}) & \cdots & P_{2N-1}(v_{N+1}) \\[1mm]
\vdots & \vdots & \vdots& \vdots \\[1mm]
P_1(v_{2N}) & P_3(v_{2N}) & \cdots & P_{2N-1}(v_{2N})
\end{pmatrix}.
$$
\begin{lemma}
	The matrix $S_u$ is invertible. 
\end{lemma}
\begin{proof}
	According to the recursion relation $vP_k=\alpha_{k+1}P_{k+1}+\alpha_kP_{k-1}$ and the fact $v_k\neq 0~(1\leq k\leq N)$, it suffices to check the invertibility of the matrix 
	$$
	\begin{pmatrix}
	1 & P_2(v_{N+1}) & \cdots & P_{2N-2}(v_{N+1}) \\[1mm]
	\vdots & \vdots & \vdots & \vdots\\[1mm]
	1 & P_2(v_{2N}) & \cdots & P_{2N-2}(v_{2N})
	\end{pmatrix}.
	$$
	Thanks to the property of the even-order Legendre polynomial, we know that $P_{2k}(v)=\widehat{P}_k(v^2)$ with $\widehat{P}_k$ a $k$-th order polynomial. Then it suffices to check the invertibility of 
	$$
	\begin{pmatrix}
	1 & v_{N+1}^2 & \cdots & (v_{N+1}^2)^{2N-2} \\[1mm]
	\vdots & \vdots & \vdots & \vdots\\[1mm]
	1 & v_{2N}^2 & \cdots & (v_{2N}^2)^{2N-2}
	\end{pmatrix}.
	$$
	According to the property of standard Vandermonde matrix, we know that the last matrix is invertible and this completes the proof
	of  the lemma.
	\end{proof} 
Recall that $g_1(0)\equiv g_1(\infty) = \beta_+-\beta_-$ and $g_u(0)=\bar{R}_2\gamma$. Then the relation \eqref{BC1} becomes 
\begin{equation*}
-2WS_u
\begin{pmatrix}
1 & \\
& \bar{R}_2
\end{pmatrix}
\begin{pmatrix}
\beta_+\\
\gamma
\end{pmatrix}
=-2WS_u
\begin{pmatrix}
\beta_-\\
0
\end{pmatrix}.
\end{equation*}
The last equation is solvable since $W$, $S_u$ and $\bar{R}_2$ are all invertible, which gives the solvability of Problem 1.\\[1mm]

\textbf{Problem 2:}
To check the solvability of Problem 2, we recall the result in \cite{MR1722195} which gives: (1) $B_2S^{-1}R_{\infty}$ is invertible if the matrix $B_2S^{-1}$ satisfies the so-called generalized Kreiss condition (GKC) proposed therein. (2) the matrix $B_2S^{-1}$ satisfies the GKC, if it satisfies the following strictly dissipative condition \cite{BenzoniSerre}:
	$$
	y^TAy<0,\qquad \text{for any}~~ y\in \text{ker} (B_2S^{-1}).
	$$
	Therefore, it suffices to check that the above strictly dissipative condition holds.

To this end, we express the kernel of $B_2S^{-1}$ as
$$
y=S\begin{pmatrix}
(n-1)I_N\\[1mm]
-I_N
\end{pmatrix}x,\quad x\in\mathbb{R}^N \setminus\{0\} .
$$
Then we compute 
$$
y^TAy = x^T\begin{pmatrix}
(n-1)I_N & -I_N
\end{pmatrix}S^TAS\begin{pmatrix}
(n-1)I_N\\[1mm]
-I_N
\end{pmatrix}x.
$$
Using Lemma \ref{S}, we have
$$
S^TAS=\begin{pmatrix}
W^{-1} & \\
& W^{-1}
\end{pmatrix}(S^{-1}AS)=\begin{pmatrix}
W^{-1} & \\
& W^{-1}
\end{pmatrix}\begin{pmatrix}
-V & \\
& V
\end{pmatrix},
$$
where $V=\text{diag}(v_{N+1},v_{N+2},...,v_{2N})$.
Thus we obtain
$$
y^TAy = -(n^2-2n)x^TW^{-1}Vx .
$$ 
Recall that $W$ and $V$ are diagonal matrices with positive entrances. Consequently, we find that $y^TAy<0$ for any $n\ge 3 $,
which gives the solvability of Problem 2 and finishes the proof of Theorem \ref{maintheorem}.

\begin{remark}
In the case $n=2$ solvability is proven as follows. 
Actually, in proving the solvability of Problem 2, the last argument 
$$
y^TAy=-(n^2-2n)x^TW^{-1}Vx<0
$$ is not true for $n=2$. 
We proceed instead as follows. We  compute $B_2S^{-1}=W(I_N,I_N)S^T$ as 
$$
B_2S^{-1}=2W\begin{pmatrix}
P_0(v_{N+1}) & 0 & P_2(v_{N+1}) & \cdots  & P_{2N-2}(v_{N+1}) & 0\\[1mm]
\vdots & \vdots & \vdots & \vdots & \vdots & \vdots \\[1mm]
P_0(v_{2N}) & 0 & P_2(v_{2N}) & \cdots &  P_{2N-2}(v_{2N})& 0
\end{pmatrix}.
$$
Then the boundary condition in Problem 2 becomes 
\begin{equation}
2WS_e\begin{pmatrix}
g_0\\
g_e
\end{pmatrix}(0)=0
\end{equation}
with 
$$
S_e=\begin{pmatrix}
P_0(v_{N+1}) & P_2(v_{N+1}) & \cdots  & P_{2N-2}(v_{N+1})\\[1mm]
\vdots & \vdots & \vdots & \vdots \\[1mm]
P_0(v_{2N})& P_2(v_{2N}) & \cdots &  P_{2N-2}(v_{2N})
\end{pmatrix}.
$$
In the proof of Lemma 6.4, we have shown that $S_e$ is invertible. Recall that $g_0(0)+\frac{\alpha_2}{\alpha_1}g_2(0)=g_0(\infty)=\beta_++\beta_-$. The boundary condition can be written as
$$
2WS_e\begin{pmatrix}
1 & -\frac{\alpha_2}{\alpha_1}e_1^T\\[2mm]
0 & I_{N-1}
\end{pmatrix}
\begin{pmatrix}
\beta_++\beta_-\\[2mm]
g_e(0)
\end{pmatrix}=0.
$$
Moreover, we use the relation $g_e(0)=\bar{R}_1\gamma$ to conclude
$$
2WS_e 
\begin{pmatrix}
1 & -\frac{\alpha_2}{\alpha_1}e_1^T\bar{R}_1\\[2mm]
0 & \bar{R}_1
\end{pmatrix}
\begin{pmatrix}
\beta_+ \\[1mm]
\gamma
\end{pmatrix} = -2WS_e \begin{pmatrix}
\beta_- \\[1mm]
0
\end{pmatrix}.
$$
The last equation is solvable since $W$, $S_e$ and $\bar{R}_1$ are invertible.
	
\end{remark}

\subsection{Numerical solution of the coupling problem}
\label{numcoupling}
Numerically, we proceed as follows. We aim at obtaining  directly  the constant $\delta$.
This is then used in  the macroscopic coupling conditions (\ref{macrocond}).
Thus we avoid solving the layer problem for each node.
Note that, from now on,  again the original ordering of the velocities is considered.

Reconsidering the invariants  (\ref{invariant}),  we have with $Z=(Z_1,Z_2,\ldots, Z_N)^T$ and (\ref{f0}) the relation 
\begin{align*}
Z= 	 B_2  S^{-1} T	
\begin{pmatrix}
D,C,\gamma
\end{pmatrix}^T			
\end{align*}
with $ B_2= (\hat I_N, (n-1)I_N) \in \R^{N\times 2N}$ as before.
Then, Gaussian elimination or a $QR$ decomposition  
transforms  $ B_2 S^{-1} T$ to the form

\begin{align*}
	\begin{pmatrix}
		1 & \delta & 0 &0&&\cdots& 0 \\
		0&1 &\delta_1&0&&\cdots&0\\
		&&\vdots&&\\
		0&\cdots &&\cdots&0&1&\delta_{N-1}
	\end{pmatrix}		\;.
\end{align*}
 In particular, we obtain directly the invariant
\begin{align*}
D +  \delta C.
\end{align*}
As discussed above,
this gives  $n-1$  equations at each  node. Together with the  balance of fluxes, which yields
$$
\sum_{i=1}^n C_i =0
$$
 we have therefore $n$
coupling conditions as required.
Additionally, we obtain  $n$ more conditions from the outgoing characteristics, i.e. equations (\ref{charac}), as before.
This gives altogether again $2n$ equations for the $2n$ unknown quantities $C^i$ and $D^i$ at each node and 
the system of macroscopic coupling conditions  (\ref{macrocond}).

\begin{remark}
For the numerical investigation and the results of the Gaussian elimination, see Section \ref{numerical1}.
\end{remark}

\begin{remark}
For a more general choice of discretization points the above  computations  can be performed in a similar way.
However, from a numerical point of view such a general choice of
points $v_i$ poses several problems. First the numerical determination
of the  eigenvectors is not as simple and efficient any more, since
the matrix $A$ is not symmetric. Second, and more important, a general
choice of discretization points (e.g., equidistant $v_i$) has the
effect that
the Vandermonde-like matrix $S$ is severely ill-conditioned  for large
$N$, see, for example,
\cite{Gautschi}.  . This results in a limited accuracy for the numerical
determination of the coupling conditions.
\end{remark}

\subsection{The   kinetic solution at  the node}
\label{nearnode}

To obtain the full  kinetic solution at the node in the limit $\epsilon \rightarrow 0 $  we have to 
determine the solution of the kinetic fixed-point problem at $x=0$. That means according to (\ref{f0}) we have to determine the values of $\gamma^i_1, \ldots, \gamma^i_{N-1}$ for each edge $i=1, \ldots, n$. That gives finally all moments of the distribution function on each edge at the node.
In particular, we obtain $\rho^i (x=0,t)$.
In case of fully symmetric coupling conditions we can simplify the procedure.
Using the above transformation of    the matrix $ B_2S^{-1} T$ we  obtain for each edge the additional $N-1$ invariants
\begin{align}
\label{invnext}
\begin{aligned}
C+ \delta_1 \gamma_1, \\
\gamma_{k-1} + \delta_{k} \gamma_{k}, \; k=2, \ldots, N-1 .
\end{aligned}
\end{align}
Moreover, we obtain directly from the coupling conditions for the odd moments
\begin{align*}
\sum_{i=1}^n g^i_{2k+1}(x=0) =0, k= 1, \ldots, N-1,
\end{align*}
which leads to 
\begin{align}
\label{invnext2}
\sum_{i=1}^n e^T_{2k} R_2^+\gamma^i =0, k= 1, \ldots, N-1 .
\end{align}
(\ref{invnext}) and (\ref{invnext2})  give the required $(N-1)(n-1)+N-1= (N-1)n$ conditions additionally to the $2n$ conditions from above and therefore 
$C^i,D^i, \gamma_1^i, \ldots,\gamma^i_{N-1}$ and thus all moments $\rho^i,q^i,g_2^i, \ldots, g_{2N-1}^i$ at $x=0$.
In particular,  $\rho^i(x=0)$ is given by 
\begin{align}
\label{rho0}
\rho^i(x=0) = 
D^i- \frac{\alpha_2 \sqrt{2}}{\alpha_1} e_1^T R_2^+ \gamma^i.
\end{align}

\subsection{Approximate  coupling conditions}
\label{halfspacemarshak}

For numerical comparison we state here the result of two approximate methods to determine the above invariant and the coefficient $\delta$, see
\cite{BK18}  for details. For further approximation methods for linear half-space problems, see \cite{M47,LF2,GK95,ST}.
Equalizing  positive half-fluxes on each edge gives 
\begin{align}
\delta = \frac{2(n-2)}{n}\;.
\end{align}
For $n=3$, we obtain 
$\delta = \frac{2}{3}$ while letting 
$n \rightarrow \infty$ gives $\delta = 2$. 
	The approach via half moment approximations of the kinetic problem from \cite{BK18} leads to 
	$$
	\delta = \frac{n-2}{n} \frac{  \frac{9}{\sqrt{3}}+4 \frac{n-2}{n}}{\frac{4}{\sqrt{3}}+2 \frac{n-2}{n}   }.
	$$
	 Here $n=3$ gives   $\delta =  \frac{1}{3} \frac{  \frac{9}{\sqrt{3}}+\frac{4}{3}}{\frac{4}{\sqrt{3}}+ \frac{2}{3}   } \sim 0.731 $
	and 
	$n=\infty$  gives $\delta = \frac{  \frac{9}{\sqrt{3}}+4 }{\frac{4}{\sqrt{3}}+2 \ } \sim 2.134$.
%

\subsection{Numerical results}
\label{numerical1}
We restrict ourselves to fully symmetric coupling conditions.
From a numerical point of view the computation of $\delta$ is independent from the solution of the network problem.
It requires in particular the knowledge of the positive eigenvalues $\lambda_i, i= 1 , \ldots , N-1$ of the matrix $A_{22}$.
Moreover, an inversion of the Vandermonde like matrix $S$  is needed and one Gaussian elimination
of $ B_2S^{-1} T$.
The matrix $S$ is well-conditioned, as long as the Gauß-Legendre points are used, see, e.g. \cite{Gautschi}.
Results are shown in Fig. \ref{fig2} (left) for the case $n=3$ and Fig. \ref{fig2} (right) for the case of infinitely many edges.
Comparing the results for large $N$ with the approximate methods in the previous section shows, in particular, the very good approximation quality of 
the half-moment approximative  method described  in detail in \cite{BK18}. 

As mentioned before, using a discrete velocity model with
equidistributed velocity discretization the Vandermonde-like matrix
$S$ tends to be severely  ill-conditioned, see, for example,
\cite{Gautschi}.  
For smaller $N$ the value of $\delta$ is approximated in this case still  in a reasonable way, however,  the solution displays oscillations for $N>20$. Note that such a behaviour  is well understood, since the condition number of the Vandermonde matrix  $S$ with the above choice of polynomials and discretization points grows exponentially with $N$ and reaches values of  order $10^{20}$ 
for $N=20$, \cite{Gautschi}.

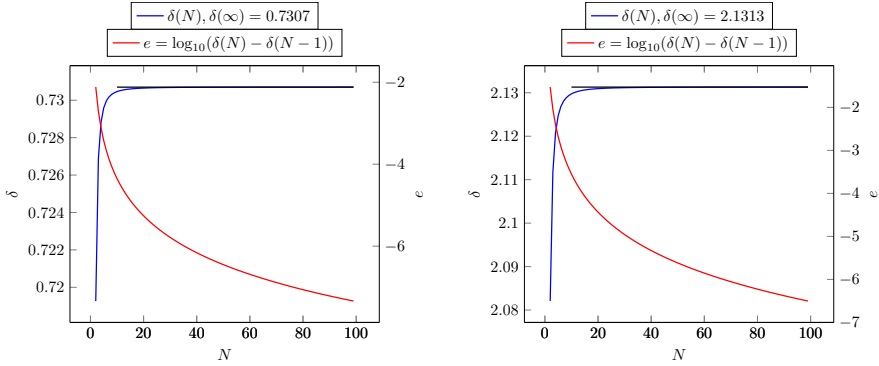
\begin{figure}[h!]
\center
	\externaltikz{funcr1111}{
	\begin{tikzpicture}[scale=0.6]

\begin{axis}[ ylabel = $\delta$, ylabel near ticks, 
xlabel =  $N$, xlabel near ticks,
legend style = {at={(0.5,1)},xshift=0.0cm,yshift=0.8cm,anchor=south},
legend columns= 1,
yticklabel style={/pgf/number format/.cd,fixed,precision=3},
axis y line*=left,
]
\addplot[color = blue, thick] file{datalegendre/datapaper/delta_legendre3.txt};
\addplot[black,thick] coordinates {(10,0.7307)  (99,0.7307)};	
\addlegendentry{$\delta(N), \delta(\infty)= 0.7307$}						
\end{axis}
\begin{axis}[ylabel =$e$,ylabel near ticks, 
legend style = {at={(0.5,1)},xshift=0.0cm,yshift=0.2cm,anchor=south},
legend columns= 1,
axis y line*=right,
]

\addplot[color = red,thick] file{datalegendre/datapaper/error3.txt};
\addlegendentry{$e=\log_{10}(\delta(N) - \delta(N-1))$}

\end{axis}
\end{tikzpicture}
}
	\externaltikz{funcr1yy111}{
	\begin{tikzpicture}[scale=0.6]

\begin{axis}[ ylabel = $\delta$, ylabel near ticks, 
xlabel =  $N$, xlabel near ticks,
legend style = {at={(0.5,1)},xshift=0.0cm,yshift=0.8cm,anchor=south},
legend columns= 1,
yticklabel style={/pgf/number format/.cd,fixed,precision=3},
axis y line*=left,
]
\addplot[color = blue, thick] file{datalegendre/datapaper/delta_legendre.txt};
\addplot[black,thick] coordinates {(10,2.1313)  (99,2.1313)};	
\addlegendentry{$\delta(N), \delta(\infty)= 2.1313$}						
\end{axis}
\begin{axis}[ylabel =$e$,ylabel near ticks, 
legend style = {at={(0.5,1)},xshift=0.0cm,yshift=0.2cm,anchor=south},
legend columns= 1,
axis y line*=right,
]

\addplot[color = red,thick] file{datalegendre/datapaper/error.txt};
\addlegendentry{$e=\log_{10}(\delta(N) - \delta(N-1))$}

\end{axis}
\end{tikzpicture}
}
	\caption{Coefficient $\delta  $ depending on $N$ for $n=3$
          (left) and $n=\infty$ (right) using Gauss-Legendre
          polynomials and points. Associated increment depending on $N$. The black line denotes the limit value $\delta  (\infty)$ of $\delta  (N)$. }
	\label{fig2}
\end{figure}

\section{A kinetic model with unbounded velocity space}
\label{unbound}
This section considers the case of a kinetic equation with  unbounded velocity space.
\subsection{Equations and coupling conditions}
\label{sec:equationsbgk}

For  $f=f(x,v,t)$ with $x\in\R$ and  $v\in\R$ at time $t\in[0,T]$ we consider the following BGK-type  model with a hyperbolic space-time scaling 
\begin{equation}\label{eq:2.1bgk}
	\partial_t f +v\partial_xf=\frac{1}{\epsilon}Q(f)= -\frac{1}{\epsilon}\left(f-\left(\rho+vq\right)M(v)\right),
\end{equation} 
where density, mean flux and total energy are given by 
$$
\rho=\int_{-\infty}^{\infty}f(v)dv, \quad q=\int_{-\infty}^{\infty}vf(v)dv
$$
and the standard Maxwellian is defined by
$$
M(v)=\frac{1}{\sqrt{2\pi}}\exp(-\frac{v^2}{2}).
$$
The associated limit equation for $\epsilon\to 0$ is the wave equation  
\begin{equation}
	\begin{aligned}
		\partial_t\rho_0+\partial_xq_0&=0, \\
		\partial_tq_0+ \partial_x\rho_0&=0.
	\end{aligned}
	\label{eq:2.2bgk}
\end{equation}

The  stationary kinetic  half-space  problem is  now 
\begin{align}\label{bgkhalfspacebg}
\begin{aligned}
v \partial_x \varphi =\frac{1}{\epsilon}Q(f)= -\frac{1}{\epsilon}\left(\varphi-\left(\rho+vq\right)M(v)\right)
\end{aligned}
\end{align}
together with the condition
$$ \int_{-\infty }^\infty \left(v - \frac{v^2}{a} \right)\varphi d v= r_-(0)= q_0(0)-a \rho_0(0).
$$
The resulting solution of the half -space problem at infinity has the form 
$$
\varphi(\infty,v) =\left(\rho_\infty +vq_\infty\right)M(v).
$$

Following  again  \cite{coron} we consider  in this case orthonormal Hermite  polynomials $P_k(v), k=0, \ldots ,2N$ on $[-\infty,\infty]$ defined by $P_0=\frac{1}{\pi^{1/4}},P_1 = \frac{\sqrt{2}}{\pi^{1/4}}v$ and
\begin{align*}
	v P_k(v) = \alpha_{k+1} P_{k+1} + \alpha_k P_{k-1}, k=1, \ldots, 2N-1
\end{align*}
with $\alpha_k = \sqrt{\frac{k}{2}}  $, compare again \cite{coron}.
Note that $$\sqrt{2}P_2= 2 v^2P_0 -P_0, \sqrt{\frac{3}{2}} P_4=  v^4 P_0 - \frac{3}{\sqrt{2}} P_2 - \frac{3}{4}P_0. $$
Define the associated functions $$H_k = P_k \exp(-\frac{v^2}{2}).$$
Using  the  transformations $v = \sqrt{2} \tilde v $  and $f= \tilde f H_0$
the kinetic equation can be rewritten as 
\begin{equation}
\label{transkin}
	\partial_t f +\sqrt{2} v\partial_xf= -\frac{1}{\epsilon}\left(f-\left(H_0 g_0 +H_1 g_1 \right)\right)
\end{equation} 
with 
$$
g_0= \int H_0(v)  f(v) dv =\frac{\rho}{\sqrt{2}}, \; g_1 =  \int H_1(v)  f(v) dv =\frac{q}{\sqrt{2}}.
$$

%
For the coupling conditions for the kinetic equation and for   the macroscopic equations we proceed as in the previous section.
However,  the kinetic coefficient  $\delta$
is different due to the change of the underlying kinetic model.

\subsection{The discrete velocity model}
\label{sec:dvmbgk}
Proceeding as before  we discretize the BGK-equation (\ref{transkin}) in velocity space and obtain a kinetic discrete velocity model  for the discrete distribution functions 
$f_i(x,t), i= 1, \ldots,  2N$ as 

\begin{align}\label{eq:kineticbgk}
	\begin{aligned}
		\partial_t f_i + \sqrt{2} v_i\partial_xf_i &= -\frac{1}{\epsilon}\left(f_i-M_i\right)
	\end{aligned}
\end{align}
with  the symmetric velocity discretization $$-\infty < v_1< v_2 < \cdots < v_N < 0 < v_{N+1}< \cdots < v_{2N-1} < v_{2N} < \infty.$$
We choose $v_i, i=1, \ldots, 2N$ to be the Gauß-Hermite  points on $[-\infty,\infty]$ and $w_i$ the associated Gauss-Hermite weights.
Defining   the moments
\begin{align*}
g_j = \sum_{i=1}^{2N}H_j(v_i) f_i\ , j	= 0, \ldots, 2N-1
\end{align*}  
the discrete  linearized Maxwellian $M_i$ is given by
\begin{align}
	M_i = w_i e^{v_i^2}( H_0 (v_i)g_0 + H_1(v_i ) g_1 ).
\end{align}
The choice of discrete Maxwellian yields for $k=0,1$
\begin{align}
	\label{mommaxbgk}
	\sum_{i=1}^{2N}M_i H_k(v_i) = g_k
\end{align}
and
\begin{align}
\sum_{i=1}^{2N}M_i H_k(v_i) = 0, k=2, \ldots,2N-1
\end{align}


due to discrete orthogonality.
%
Moreover,  with the present choice of discretization points and polynomials we have
\begin{align*}
\sum_{i=1}^{2N}H_{2N} (v_i)f_i\ =0.
\end{align*}
Let now $G=(u,g)^T$ with $u=(g_0,g_1)^T$ and   $g = (g_2,\ldots,g_{2N-1})$ be defined as before 
and consider   the  Vandermonde like  matrix
\begin{align*}
	S
	=\begin{pmatrix}
		P_0(v_1) &\cdots&P_0(v_{2N})\\
		\vdots&	&\vdots\\
		P_{2N-1}(v_1)&\cdots&P_{2N-1}(v_{2N})\\
	\end{pmatrix}	 \in \R^{2N \times 2N}
\end{align*}
together with the matrix 
$E=\text{diag}(e^{-v_1^2/2},\cdots,e^{-v_{2N}^2/2})  $. Then, 
  the  variables $f$ are transformed into   the moments $ SEf = G.$ 

Using the recursion formula of the Hermite polynomials and the above remarks,
the kinetic equation is rewritten in moment variables as
\begin{align}\label{eq:macro_4eq_hbgk}
	\begin{aligned}
		\partial_t g_0 + \sqrt{2} \alpha_1 \partial_x g_1 &= 0\\
		\partial_t g_1+ \sqrt{2} \partial_x (\alpha_2 g_2 + \alpha_1g_0 )&= 0\\
		\partial_t g_{2} + \sqrt{2}\partial_x (\alpha_{3} g_{3} + \alpha_2 g_1) &=  - \frac{1}{\epsilon}g_{2}\\
		\partial_t g_{k} + \sqrt{2} \partial_x (\alpha_{k+1}g_{k+1} + \alpha_{k} g_{k-1})&= - \frac{1}{\epsilon}g_{k} , k=3, \ldots, 2N-2\\
		\partial_t g_{2N-1}+ \sqrt{2}\partial_x (\alpha_{2N-1}g_{2N-2})&= - \frac{1}{\epsilon}g_{2N-1} \end{aligned}
\end{align}
%
and renaming gives for the first 3 equations
\begin{align}
	\begin{aligned}
		\partial_t \rho+ \partial_x q&= 0\\
		\partial_t q+   \partial_x (2 g_2 +\rho )&= 0\\
		\partial_t g_{2} + \sqrt{2}\partial_x (\alpha_{3} g_{3} + \alpha_2 \frac{q}{\sqrt{2}}) &=  - \frac{1}{\epsilon}g_{2}.\end{aligned}
\end{align}
Again  this system leads in  the limit $\epsilon \rightarrow 0$ directly to  the wave equation
(\ref{eq:2.2bgk}).

\subsection{The discrete layer problem and coupling conditions}
\label{sec:disclayerbgk}

The corresponding discrete kinetic  layer equation is
\begin{align}
	\begin{aligned}
	 \sqrt{2} \alpha_1 \partial_x g_1 &= 0\\
 \sqrt{2} \partial_x (\alpha_2 g_2 + \alpha_1g_0 )&= 0\\
 \sqrt{2}\partial_x (\alpha_{3} g_{3} ) &=  - g_{2}\\
		\sqrt{2} \partial_x (\alpha_{k+1}g_{k+1} + \alpha_{k} g_{k-1})&= - g_{k} , k=3, \ldots, 2N-2\\
		\sqrt{2}\partial_x (\alpha_{2N-1}g_{2N-2})&= -g_{2N-1}.
	\end{aligned}
\end{align}

One obtains   in terms of  $\rho$ and $q$ that $q=C $ and $D= 2 g_2 +\rho$ 
for constants $C \in \R$ and $D \in \R^+$.

 Moreover, in matrix form the equations for $g_2, \ldots, g_{2N-1}$ are given as in the previous section  by the linear system 
$$
\sqrt{2} \partial_x g = - A_{22}^{-1} g 
$$
with the symmetric tridiagonal matrix $A_{22} \in \R^{2(N-1) \times
  2(N-1)}$ as in \eqref{eq:A_tridiag}, of course containing the 
values of $\alpha_3\ldots,\alpha_{2N-1}$ associated to the Hermite polynomials. 
As previously the matrix of  eigenvectors of $A_{22}$ associated to  positive eigenvalues is denoted by $R_2^+$.

For the analytical solution of the coupling problem, note that after reordering the velocities as
$(v_N, \ldots, v_1, v_{N+1}, \ldots , v_{2N})$ we have with reordered quantities  $G=S E f$ where 
	$$
	E=\begin{pmatrix}
	\bar{E} & \\
	& \bar{E}
	\end{pmatrix},\quad \bar{E}=\text{diag}(e^{-v_{N+1}^2/2},\cdots,e^{-v_{2N}^2/2}).
	$$ 
	$S$ is defined by the same expression as \eqref{VS} with  $P_k$ being the orthonormal Hermitian polynomials. Thanks to the expression of $B_1$ and $B_2$, we know that $B_1E^{-1}=\bar{E}^{-1}B_1$ and $B_2E^{-1}=\bar{E}^{-1}B_2$. Therefore, the coupling condition are, as in \eqref{coupling-DV}, given by 
	$$
	\mathcal{B}
	\begin{pmatrix}
	 S^{-1}G^1(0,t),
	 S^{-1}G^2(0,t),
	\ldots
	 S^{-1}G^n(0,t)
	\end{pmatrix}^T=0.
	$$
%
Moreover, we have again
		$$
		S^{-1}=\begin{pmatrix}
			W & \\
			& W
		\end{pmatrix}S^T,
		$$
	where $W$ is an $N\times N$ diagonal matrix with 
	positive entrances given here by 
	 the Gaussian-Hermite quadrature weights  $w_{N+1},w_{N+2},...,w_{2N}$.
Then, the  proof proceeds exactly along the same lines as before.

For the numerical determination of the macroscopic  invariants we  compute the matrix $T \in \R^{(2N)  \times (N+1)}$ as 
%

\begin{align*}
	T=	\begin{pmatrix}
		T_{11} & T_{12}\\[1mm]
		0 & R_2^+ 
	\end{pmatrix}
\end{align*}
with
\begin{align*}
	T_{11}= \frac{1}{\sqrt{2}} \begin{pmatrix}
		1&0  \\
		0&1
	\end{pmatrix}, \quad
	T_{12}=	- \sqrt{2}\begin{pmatrix}
		e_1^T R_2^+ \\
		0
	\end{pmatrix},	
\end{align*}
where $e_1^T =(1,0,\ldots,0)$ is the unit vector.
Using these matrices and proceeding exactly as in the previous section, 
we obtain for symmetric nodes  the invariants
\begin{align*}
	D +  \delta C.
\end{align*}
This gives  as discussed above $n-1$  equations at each  node. Together with the equality of fluxes 
$$
\sum_{i=1}^n C^i =0$$
we have therefore $n$
coupling conditions as required.
Moreover, as in the bounded case we can compute the  values of all moments at the boundary and in particular 
\begin{align}
	\label{rho00}
	\rho^i(x=0) = D^i-2 e_1^T R_2^+ \gamma^i  .
\end{align}

In the section on numerical results we also compute an approximation of the kinetic distribution functions at the node
on all edges using the Hermite expansion.

That means, in this case, we compute  $f^i=f^i(x=0,v) $ for $i=1,2,3$ and $v \in \R$ by 

\begin{align}
	\label{spectraldist}
f(v) = H_0(\frac{v}{\sqrt{2}})   \sum_{k=0}^{2N-1} g_k H_k (\frac{v}{\sqrt{2}}),
\end{align}
where
\begin{align*}
g_0= \frac{D}{\sqrt{2}}, g_1 =\frac{C}{\sqrt{2}}
\end{align*}
and for $k=2, \ldots, 2N-1$ 
\begin{align*}
g_k =e_{k-1}^T R_2^+ \gamma.
\end{align*}

\subsection{Approximate coupling conditions}
\label{halfspacemarshakbgk}

Equalizing  the  positive half-fluxes on each edge
gives here 
\begin{align}
	\delta = \frac{\sqrt{ \pi} (n-2)}{\sqrt{2}n}
\end{align}
For example for $n=3$, we obtain for  the  factor 
$\delta = \frac{\sqrt{ \pi} }{\sqrt{2} 3} \sim 0.4178$. 
$n \rightarrow \infty$ gives $\delta = \frac{\sqrt{ \pi}}{\sqrt{2}} \sim 1.253$.
	The approach via half moment approximations of the kinetic problem from \cite{BK18} leads to 
	$$
	\delta = \frac{n-2}{n} \frac{ 4+\frac{n-2}{n}\sqrt{2\pi}}{\sqrt{2\pi}+2 \frac{n-2}{n}   }
	$$

	$n=3$ gives here   $\delta =  \frac{1}{3} \frac{
          4+\frac{1}{3}\sqrt{2\pi}}{\sqrt{2\pi}+ \frac{2}{3}    } \sim
        0.5079 $, while
	$n=\infty$ gives $\delta =  \frac{ 4+\sqrt{2\pi}}{\sqrt{2\pi}+2    }\sim 1.4438$.


\subsection{Numerical results}
\label{numericalbgk}
As in Section \ref{numerical1} we restrict ourselves to fully symmetric coupling conditions.
Using the Vandermonde like matrix $S$ in a naive way the  problem is ill-conditioned for large $N$ although normalized Hermite polynomials and the associated points are used, \cite{Gautschi}.  This problem can be removed by  using a simple rescaling of $S$.
Numerical results are shown in Fig. \ref{fig2.2} (left) for the case $n=3$ and Fig. \ref{fig2.2} (right) for the case of infinitely many edges.
Further numerical experiments, for example for $N=3000$, did achieve an error increment of the order $e \sim 10^{-9}$.
Comparing the results for large $N$ with the approximate methods in the previous section shows again  the very good approximation quality of 
the half-moment approximation   given in detail in \cite{BK18}.

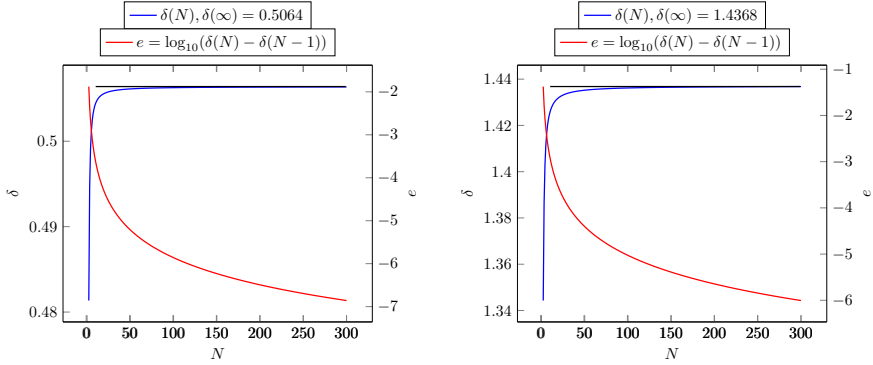
\begin{figure}[h!]
	\center
	\externaltikz{funcr12211}{
	\begin{tikzpicture}[scale=0.6]

	\begin{axis}[ ylabel = $\delta$, ylabel near ticks, 
	xlabel =  $N$, xlabel near ticks,
	legend style = {at={(0.5,1)},xshift=0.0cm,yshift=0.8cm,anchor=south},
	legend columns= 1,
	yticklabel style={/pgf/number format/.cd,fixed,precision=3},
	axis y line*=left,
	]
	\addplot[color = blue, thick] file{datahermite/datapaper/delta_hermite3.txt};
	\addplot[black,thick] coordinates {(10,0.5064)  (299,0.5064)};	
	\addlegendentry{$\delta(N), \delta(\infty)= 0.5064$}						
	\end{axis}
	\begin{axis}[ylabel =$e$,ylabel near ticks, 
	legend style = {at={(0.5,1)},xshift=0.0cm,yshift=0.2cm,anchor=south},
	legend columns= 1,
	axis y line*=right,
	]
	
	\addplot[color = red,thick] file{datahermite/datapaper/error_hermite3.txt};
	\addlegendentry{$e=\log_{10}(\delta(N) - \delta(N-1))$}
	
	\end{axis}
	\end{tikzpicture}
}
	\externaltikz{funcr122yy11}{
	\begin{tikzpicture}[scale=0.6]

	\begin{axis}[ ylabel = $\delta$, ylabel near ticks, 
	xlabel =  $N$, xlabel near ticks,
	legend style = {at={(0.5,1)},xshift=0.0cm,yshift=0.8cm,anchor=south},
	legend columns= 1,
	yticklabel style={/pgf/number format/.cd,fixed,precision=3},
	axis y line*=left,
	]
	\addplot[color = blue, thick] file{datahermite/datapaper/delta_hermite.txt};
	\addplot[black,thick] coordinates {(10,1.4368)  (299,1.4368)};	
	\addlegendentry{$\delta(N), \delta(\infty)= 1.4368$}						
	\end{axis}
	\begin{axis}[ylabel =$e$,ylabel near ticks, 
	legend style = {at={(0.5,1)},xshift=0.0cm,yshift=0.2cm,anchor=south},
	legend columns= 1,
	axis y line*=right,
	]
	
	\addplot[color = red,thick] file{datahermite/datapaper/error_hermite.txt};
	\addlegendentry{$e=\log_{10}(\delta(N) - \delta(N-1))$}
	
	\end{axis}
	\end{tikzpicture}
}
	\caption{Coefficient $\delta  $ depending on $N$ for $n=3$ (left) and $n=\infty$ (right) using Hermite  polynomials and points. Associated increment depending on $N$. The black line denotes the limit value $\delta  (\infty)$ of $\delta  (N)$.}
	\label{fig2.2}
\end{figure}

\section{Numerical comparison of solutions on the network}
\label{network}


%
%
%
%

To illustrate the above results,  we consider the case $v \in \R$ and a single node  with 3 edges.
As initial conditions for the kinetic equation we choose equilibrium distributions $f^i(x,v) = \rho^i (x) M(v)$,
with  macroscopic densities $\rho^1=1$, $\rho^2=0$ and $\rho^3=2$. The resulting fluxes are   $q^j = 0$ $j=1,\dots,3$.
These data are also prescribed at the  outer boundaries.

In Figure \ref{fig_n1}  on the left  the densities $\rho^i$ on  the three edges are   displayed at time $t=0.1$.
The kinetic solution is computed by a standard Finite-difference scheme and shown for  $\epsilon= 10^{-1}, \epsilon= 10^{-2}$ and   $\epsilon= 5 \cdot  10^{-3}$. In the right figure a zoom to the solution on edge 2 
is shown.
Up to   kinetic  layers of order $\mathcal{O} (\epsilon) $ we observe a very good agreement of the half-moment and spectral coupling with the kinetic  model. 
Also the  approximation via half-fluxes is relatively close to the kinetic results with a deviation of approximately $10\%$.
The value of the density of the  kinetic solution at the node determined by the spectral method (\ref{rho00}) is shown with a red marker and agrees very well with the Finite-Difference kinetic solution at the node.

In Figure \ref{fig_n2}  on the left a  further vertical  zoom  is shown for  the density on edge  $2$. 
The kinetic solution is   shown for  $ \epsilon= 10^{-2}$ and   $\epsilon= 5 \cdot 10^{-3}$. 
On this scale the deviation of the spectral solution from the solution obtained from the half moment approximation is clearly seen.
In the right figure a  zoom to the solution on edge 2 near  the node 
is shown displaying the  kinetic layer near the node in more detail.

 Figure \ref{fig_n3}  on the left shows the kinetic distribution functions on all edges computed by the Finite-Difference method with 
  $\epsilon= 5 \cdot 10^{-4}$ and $\Delta x = 10^{-4}$ and by the spectral method with $N=1000$ using (\ref{spectraldist}). Near the discontinuity of the distribution function Gibbs oscillations are observed for the spectral method as expected. On the right a  zoom to the kinetic solution  on edge $2$  is shown computed by FD and the spectral method  with and without a Fejer-type filter.  

\begin{figure}											
	\externaltikz{net12}{
		\begin{tikzpicture}[scale=0.7]
			\begin{axis}[ylabel = $\rho$,xlabel = $x$,
				legend style = {at={(0.5,1)},xshift=0.2cm,yshift=-0.0cm,anchor=south},
				legend columns= 1,
				xmin = 0., xmax = 0.3	
				]

					\addplot[color = brown,thick] file{data_matlab/rho_kinetic_1_0001.txt};
				\addlegendentry{Kinetic $1$ }
				
					\addplot[color = blue,thick] file{data_matlab/rho_kinetic_2_01.txt};
						\addlegendentry{Kinetic $2$}
				\addplot[color = green,thick] file{data_matlab/rho_kinetic_3_01.txt};
		\addlegendentry{Kinetic $3$}
			\addplot[mark=none, thick,cyan, domain = 0:0.1] {1-0.4178*1/(1+0.4178*1)};
		\addlegendentry{Half-fluxes $2$}
		
		\addplot[mark=none, black, thick, domain = 0:0.1] {1-0.5064*1/(1+0.5064*1)};
		\addlegendentry{Spectral/Half-moment $2$}
		
				\addplot[color = blue!,thick] file{data_matlab/rho_kinetic_2_001.txt};
				\addplot[color = blue,thick] file{data_matlab/rho_kinetic_2_0005.txt};
				
				\addplot[color = green,thick] file{data_matlab/rho_kinetic_3_001.txt};
					\addplot[color = green,thick] file{data_matlab/rho_kinetic_3_0005.txt};
				
			
			

				\addplot[mark=none, black, thick, domain = 0.1:0.3] {0};
			
				
			\end{axis}
		\end{tikzpicture}
			}	
		\externaltikz{net233}{
		\begin{tikzpicture}[scale=0.7]
	\begin{axis}[ylabel = $\rho$,xlabel = $x$,
	legend style = {at={(0.5,1)},xshift=0.2cm,yshift=-0.0cm,anchor=south},
	legend columns= 1,
	xmin = 0.0, xmax = 0.12,
	ymin = 0.5, ymax=0.81	
	]
	

	\addplot[color = blue,thick] file{data_matlab/rho_kinetic_2_01.txt};
\addlegendentry{Kinetic $2$}

	\addplot[mark=none, thick,cyan,domain = 0:0.1] {1-0.4178*1/(1+0.4178*1)};
\addlegendentry{Half-fluxes $2$}
\addplot[mark=none, magenta, thick, domain = 0:0.1] {1-0.5079*1/(1+0.5079*1)};

\addlegendentry{Half moment $2$}

	 \addplot[mark=none, black, thick, domain = 0:0.1] {1-0.5064*1/(1+0.5064*1)};
\addlegendentry{Spectral $2$}

	\addplot[mark=none, thick,red,domain = 0:0.005] {0.80036};
\addlegendentry{Spectral $\rho^2(x=0) $}

	\addplot[color = blue,thick] file{data_matlab/rho_kinetic_2_001.txt};

	\addplot[color = blue,thick] file{data_matlab/rho_kinetic_2_0005.txt};


	\addplot[mark=none, black, thick, domain = 0.1:0.3] {0};

	\end{axis}
	\end{tikzpicture}
}	
	\caption{$\rho$ for all edges, kinetic solution for  $\epsilon= 10^{-1}, \epsilon= 10^{-2}$ and   $\epsilon= 5 \cdot 10^{-3}$ at time $t=0.1$ (left). Zoom to solution on edge 2 (right).}
	\label{fig_n1}
\end{figure}
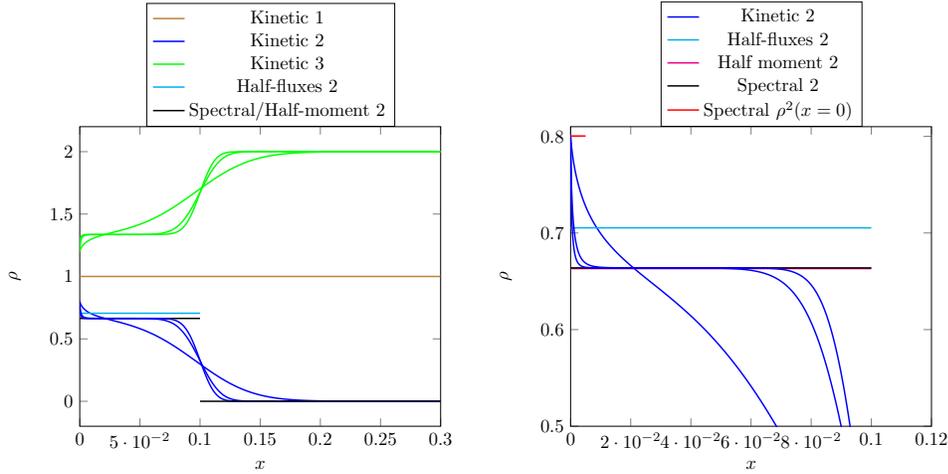

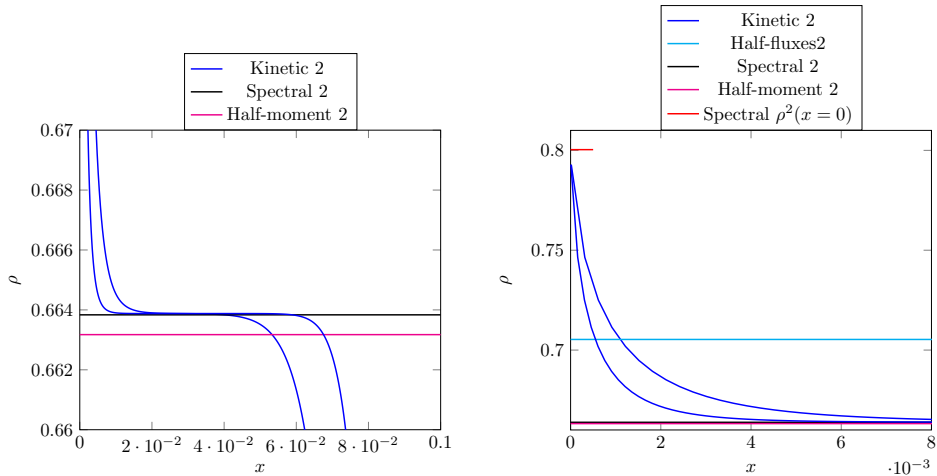
\begin{figure}											
	\externaltikz{net33}{
		\begin{tikzpicture}[scale=0.7]
		\begin{axis}[ylabel = $\rho$,xlabel = $x$,
			legend style = {at={(0.5,1)},xshift=0.2cm,yshift=-0.0cm,anchor=south},
			legend columns= 1,
			xmin = 0.0, xmax = 0.1,
			ymin = 0.66, ymax=0.67,
	  y tick label style={/pgf/number format/precision=4},
			 ]

			
%
%
				\addplot[color = blue,thick] file{data_matlab/rho_kinetic_2_001.txt};
			\addlegendentry{Kinetic $2$};

			\addplot[mark=none, black, thick, domain = 0:0.1] {1-0.5064*1/(1+0.5064*1)};
			\addlegendentry{Spectral $2$}

			\addplot[mark=none, magenta, thick, domain = 0:0.1] {1-0.5079*1/(1+0.5079*1)};
			\addlegendentry{Half-moment $2$}

			\addplot[color = blue,thick] file{data_matlab/rho_kinetic_2_0005.txt};

		\end{axis}
	\end{tikzpicture}
}
	\externaltikz{net4}{
		\begin{tikzpicture}[scale=0.7]
	\begin{axis}[ylabel = $\rho$,xlabel = $x$,
	legend style = {at={(0.5,1)},xshift=0.2cm,yshift=-0.0cm,anchor=south},
	legend columns= 1,
	xmin = 0.0, xmax = 0.008,
	ymin = 0.66, ymax=0.81,
	 y tick label style={/pgf/number format/precision=4},
	]


	\addplot[color = blue,thick] file{data_matlab/rho_kinetic_2_001.txt};
	\addlegendentry{Kinetic $2$}
	
		\addplot[mark=none, thick,cyan,domain = 0:0.1] {1-0.4178*1/(1+0.4178*1)};
	\addlegendentry{Half-fluxes$2$}
	
	\addplot[mark=none, black, thick, domain = 0:0.1] {1-0.5064*1/(1+0.5064*1)};
	\addlegendentry{Spectral $2$}
	
	\addplot[mark=none, magenta, thick, domain = 0:0.1] {1-0.5079*1/(1+0.5079*1)};
	\addlegendentry{Half-moment $2$}
	
		\addplot[mark=none, thick,red,domain = 0:0.0005] {0.80036};
	\addlegendentry{Spectral $\rho^2(x=0) $}
	\addplot[color = blue,thick] file{data_matlab/rho_kinetic_2_0005.txt};

%
	\end{axis}
	\end{tikzpicture}
			}
	\caption{$\rho$ for  edge $2$, kinetic solution for  $\epsilon= 10^{-2}$  and  $\epsilon= 5 \cdot 10^{-3}$  at time $t=0.1$, vertical zoom(left). Horizontal zoom to the kinetic layer near the node (right).}
	\label{fig_n2}
\end{figure}

\begin{figure}											
	\externaltikz{net51}{
		\begin{tikzpicture}[scale=0.7]
			\begin{axis}[ylabel = $f$,xlabel = $v$,
				legend style = {at={(0.5,1)},xshift=0.2cm,yshift=-0.0cm,anchor=south},
				legend columns= 2,
				xmin = -5, xmax = 5,
				y tick label style={/pgf/number format/precision=4},
				]

				\addplot[color = brown!50,thick] file{data_matlab/f_hermite1.txt};
				\addlegendentry{Spectral $1$}
				
				\addplot[color = brown,thick] file{data_matlab/f_kinetic_1.txt};
				\addlegendentry{Kinetic $1$}
				\addplot[color = blue!50,thick] file{data_matlab/f_hermite2.txt};
				\addlegendentry{Spectral $2$}
					\addplot[color = blue,thick] file{data_matlab/f_kinetic_2.txt};
				\addlegendentry{Kinetic $2$}
				\addplot[color = green!50,thick] file{data_matlab/f_hermite3.txt};
				\addlegendentry{Spectral $3$}

				\addplot[color = green,thick] file{data_matlab/f_kinetic_3.txt};
				\addlegendentry{Kinetic $3$}
				
					\addplot [thick] coordinates {(0.0,0.0) (0.0,0.6)};
				%
				%

			
			\end{axis}
		\end{tikzpicture}
	}
		\externaltikz{net71}{
		\begin{tikzpicture}[scale=0.7]
			\begin{axis}[ylabel = $f$,xlabel = $v$,
				legend style = {at={(0.5,1)},xshift=0.2cm,yshift=-0.0cm,anchor=south},
				legend columns= 1,
				xmin = -0.6, xmax = 0.6,
				ymin = 0.2, ymax=0.5,
				y tick label style={/pgf/number format/precision=4},
				]

				\addplot[color = brown,thick] file{data_matlab/filtere/f_hermite2_0_e.txt};
				\addlegendentry{Spectral unfiltered}
				\addplot[color = red,thick] file{data_matlab/filtere/f_hermite2_2_e.txt};
				\addlegendentry{Spectral filtered}
				\addplot[color = blue,thick] file{data_matlab/f_kinetic_2.txt};
				\addlegendentry{Kinetic $2$}
				
				\addplot [thick] coordinates {(0.0,0.0) (0.0,0.6)};
				%

				
			\end{axis}
		\end{tikzpicture}
		
	}
	
	\caption{Left: Kinetic solutions for all edges at node $x=0$
          at time $t=0.1$ computed by FD method with   $\epsilon= 5
          \cdot 10^{-3}$ and $\Delta x = 10^{-4}$ and by spectral
          method with $N=1000$. Right: Zoom to kinetic solution  on edge $2$  computed by FD and spectral method  with and without  filtering. }
\label{fig_n3}
\end{figure}
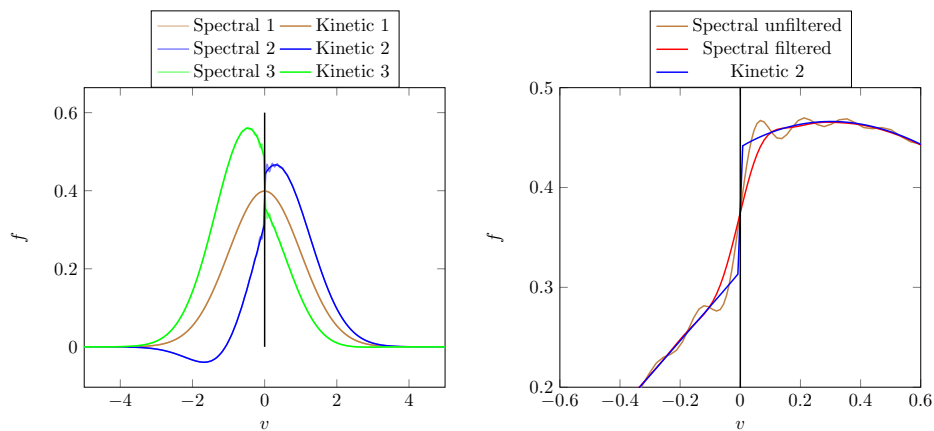

\section{Conclusion and Outlook}
\label{conclusion}

In this work we have considered the derivation of coupling conditions for a macroscopic equation on networks from the underlying kinetic equations and conditions.
In particular, we have discussed here the case of a kinetic linearized BGK type model and the associated 
wave equation.
The procedure is based on an asymptotic analysis  of the situation near the nodes and the investigation of 
the kinetic layer near the nodes and the associated coupled kinetic half-space problems. 
For the numerical solution a very accurate  spectral procedure to determine  the macroscopic coupling conditions has been developed. From the analytical side we have proven well-posedness of the coupled half-space problems for general  BGK-type discrete velocity models.

The approach can be extended to more  complicated problems like the full BGK model with the  linearized Euler equations as limit equations.  The investigation requires, additionally to the discussion of the kinetic half-space problems, also the investigation of related viscous layers.
This will be considered in a forthcoming publication.

The validity of a higher order asymptotic expansion of the kinetic coupling problem and a  rigorous proof of convergence of the kinetic solution on the network towards the macroscopic 
solution
 will be also considered in future work following the general approach developed in  \cite{ZY22}.

Finally, we mention, that codes and data that allow readers to reproduce the most important numerical results, in particular the determination of the coupling coefficients,  are available at 

https://gitlab.rhrk.uni-kl.de/klar/kinetic-network.git

\end{document}